\documentclass[12pt]{amsart}
\usepackage{a4}
\usepackage{latexsym}
\usepackage{amsmath}
\usepackage{amssymb}
\usepackage{amscd}
\usepackage{color}
\usepackage{amsthm}
\usepackage{hyperref}

\theoremstyle{plain}
    \newtheorem{theorem}                    {Theorem}[section]  
    \newtheorem{lemma}      [theorem]       {Lemma}
    \newtheorem{corollary}  [theorem]       {Corollary}
    \newtheorem{proposition}[theorem]       {Proposition}

    \newtheorem{question}       {Question}

\newtheorem{example}[theorem]{Example}

\newtheorem{remark}[theorem]{Remark}

\DeclareFontFamily{U}{wncy}{}
    \DeclareFontShape{U}{wncy}{m}{n}{<->wncyr10}{}
    \DeclareSymbolFont{mcy}{U}{wncy}{m}{n}
    \DeclareMathSymbol{\Sh}{\mathord}{mcy}{"58} 

\newcommand{\Gr}{\operatorname{Gr}}

\newcommand{\Hom}{\operatorname{Hom}}

\newcommand{\Ext}{\operatorname{Ext}}

\newcommand{\Pic}{\operatorname{Pic}}

\newcommand{\Hilb}{{\operatorname{Hilb}}}

\newcommand{\Spec}{\operatorname{Spec}}
\newcommand{\trdeg}{\operatorname{trdeg}}

\newcommand{\id}{\operatorname{id}}
\newcommand{\Br}{\operatorname{Br}}
\newcommand{\CH}{\operatorname{CH}}
\renewcommand{\lim}{\operatornamewithlimits{lim}}
\newcommand{\colim}{\operatornamewithlimits{colim}}

\newcommand{\NS}{\operatorname{NS}}
\newcommand{\f}{{\mathcal F}}
\renewcommand{\O}{{\mathcal O}}

\newcommand{\Z}{{{\mathbb Z}}}

\newcommand{\Q}{{{\mathbb Q}}}

\renewcommand{\P}{{{\mathbb P}}}
\newcommand{\F}{{{\mathbb F}}}

\newcommand{\G}{{{\mathbb G}}}

\newcommand{\K}{{\mathbb K}}

\newcommand{\M}{{\mathcal M}}
\newcommand{\A}{{\mathbb A}}
\newcommand{\et}{{\text{\rm et}}}

\newcommand{\Zar}{{\text{\rm Zar}}}

\newcommand{\Div}{\operatorname{Div}}
\renewcommand{\div}{\operatorname{div}}

\newcommand{\fiber}{\operatorname{fiber}}

\newcommand{\proofend}{\hfill$\square$\\ \smallskip}
\newcommand{\GG}{{\mathcal G}(p^\infty)}

\date{\today}
\title{On Frobenius Rigidity for Motivic Cohomology}
\author{Thomas H. Geisser}
\address{Department of Mathematics, Rikkyo University, Ikebukuro, Tokyo, Japan}
\email{geisser@rikkyo.ac.jp}
\thanks{Supported by JSPS Grant-in-Aid (B) 23K25768\\
\href{https://orcid.org/0000-0002-1153-4097}
{ORCID: 0000-0002-1153-4097}
}
\subjclass[2020]{Primary:\ 14G15; Secondary:\ 11G10,\ 11G25,\ \ 14F22,\ 14K15}
\keywords{Motivic cohomology, Weil-\'etale cohomology, Frobenius morphism, rigidity, algebraically closed fields}

\begin{document}

\begin{abstract}
We study how motivic and \'etale motivic cohomology of a smooth and
proper variety $X$ changes under extensions of algebraically closed base
fields $k$. We show that with finite coefficients, motivic cohomology
is independent of $k$, and \'etale motivic cohomology is independent
of $k$ away from the characteristic, but depends on $k$ 
at the characteristic. 

We then consider the fiber of $F-1$ for $F$ the Frobenius, on 
motivic cohomology
and \'etale motivic cohomology for varieties defined over finite fields, 
i.e., Weil-\'etale cohomology. 
Combining the results of the first part with the structure
theory of perfect unipotent group schemes, we show that this fiber is
independent of $k$ with finite coefficients. Finally, we give some results
and examples with integral coefficients.
\end{abstract}

\maketitle


\section{Introduction}
The aim of this article is to study the rigidity of motivic and
\'etale cohomology, i.e., the question
whether these theories change under an extension of algebraically closed
fields. Asked in this naive form, one can see, for example, by looking
at the units that they
depend on the algebraically closed field. But we consider two more refined
versions of this question. The
first version concerns motivic and \'etale cohomology with finite
coefficients.

\medskip 

\noindent{\it
Let $X/k$ be a  smooth and proper variety over an algebraically
closed field $k$ and $\K$ an algebraically closed extension field. Are the groups 
$$H^i_\M(X_\K,\Z/m(n)), \quad H^i_\et(X_\K,\Z/m(n))$$
independent of $\K$, and are they finite?}

\medskip

By independence of $\K$ we mean that the canonical base change map 
induces an isomorphism. 
Our findings can be summarized as follows: 

For all $m$, the motivic cohomology groups $H^i_\M(X_\K,\Z/m(n))$
are independent of $\K$ by an analogue of Suslin's rigidity 
theorem. 
They are known to be finite for $n=0,1,\dim X$, but Schoen 
gave an example in which $H^4_\M(X_\K,\Z/m(2))$ is infinite 
for a three-fold
in characteristic zero. No such example is known in characteristic $p>0$.

The \'etale motivic cohomology groups 
$H^i_\et(X_\K,\Z/m(n))$ are independent of $\K$ and finite
for $m$ invertible
by the identification $\Z/m(n)\cong \mu_m^{\otimes n}$ 
and the smooth and proper base change theorem. For $p$ the
exponential characteristic of $k$, 
the groups
$H^i_\et(X_\K,\Z/p(n))$ do depend on $\K$ and can be infinite
(except for $n=0$ and $n=\dim X$). 

With integral coefficients, independence of $\K$ fails
(for example, $H^1_\M(\K,\Z(1))\cong \K^\times$),
but one can ask if, for a variety
defined over a finite field, the fiber of $F-1$ for $F$ the Frobenius,
on cohomology is independent of the algebraically closed field.
This question is motivated by applications to
Weil-\'etale cohomology, and has been raised previously by several people. 
More precisely, for a smooth and proper $X/\F_q$ and 
an algebraically closed field $\F_q\subset \K$,  we define
\begin{align*}
H^i_W(X_\K,\Z(n)) &= H^i\Big( \fiber: 
R\Gamma_\et(X_\K,\Z(n))\stackrel{F-1}{\longrightarrow}
R\Gamma_\et(X_\K,\Z(n))\Big)\\
H^i_F(X_\K,\Z(n)) &= H^i\Big( \fiber: 
R\Gamma_\M(X_\K,\Z(n))\stackrel{F-1}{\longrightarrow}
R\Gamma_\M(X_\K,\Z(n))\Big),
\end{align*}
where $F$ acts on $\K$.
The former, Weil-\'etale cohomology, was introduced by 
Lichtenbaum \cite{lichtenbaum}
with the goal of giving a cohomological
interpretation of special values of zeta functions. The theory was further
developed in \cite{Geisser04}. The latter groups were studied in \cite{ichfrob}.
Independence of the algebraically closed field would provide 
additional flexibility, 
for example, there could be better geometric results available over
larger fields. 
Our main result is the following theorem:

\begin{theorem}
For all $i,n\geq 0$ and $m\geq 1$, 
the groups $H^i_W(X_\K,\Z/m(n))$ are finite and do not depend on $\K$.
\end{theorem}

For integral coefficients we obtain partial results, for example
$H^i_W(X_\K,\Z(n))$ is independent of $\K$ for $n=0$ and $1$. 
We furthermore discuss the case of Chow groups, relate our questions 
to Parshin's and Beilinson's conjecture, and interpret them in terms
of the cycle complex and Gersten resolutions.

There are analogous questions for Suslin homology,
higher Chow groups, K-theory, \'etale K-theory, and other theories,
which can be treated by methods similar to those used in this article.
For $\A^1$-homotopy theory some results were obtained by
Richarz and Scholbach \cite{RS}.

\smallskip

{\bf Conventions:} For a finite field $\F_q$ of characteristic 
$p$, we let $\F$ be its algebraic closure. 
The symbol $\K$ will stand
for an algebraically closed field, necessarily containing $\F$
if the characteristic of $\K$ is $p$. 

For $X$ equidimensional of finite type over a field, Bloch's cycle
complex $z^n(X,*)$ is, in degree $i$, the free abelian group on closed
integral subschemes of $X\times\Delta^i$ of codimension $n$ meeting all
faces properly, and $\CH^n(X,i)=H_i(z^n(X,*))$ \cite{bloch-cycles}.
For $X$ smooth, motivic cohomology is defined by
$R\Gamma_\M(X,\Z(n))=z^n(X,2n-*)$, i.e., $H^i_\M(X,\Z(n))=\CH^n(X,2n-i)$,
and similarly with coefficients in $\Z/m$ and $\Q$. \'Etale motivic
cohomology is the \'etale hypercohomology of the complex of sheaves
$\Z(n)=z^n(-,2n-*)$.
\smallskip

{\bf Acknowledgements:} We are grateful to
M. Morrow, K. R\"ulling, A. Schmidt, and J. Scholbach
for discussions and comments.
O. R\"ondigs pointed out to us that Suslin rigidity holds
without condition on $m$.
AI has been used to polish the final version.

\section{Injectivity and rigidity for higher Chow groups}
In general, the functors we are considering will be injective
on extensions of algebraically closed fields by the following 
lemma.

\begin{lemma}(Injectivity)\label{inject}
Let $k$ be an algebraically closed field and let 
$\f$ be a functor on essentially smooth $k$-algebras. 
Assume that $\f$ commutes with the following filtered colimits:
\begin{enumerate}
\item For any field extension $k\subseteq K$, 
$\colim_{k\subset E\subset K, f.g.}\f(E)\cong \f(K)$.
\item For smooth connected affine $S\to k$,  
$\colim_{V\subset S \text{affine,open}}\f(V)\cong\f(k(S))$.
\end{enumerate}
Then $\f(k)\to \f(K)$ is injective for any extension field
$K$ of $k$.
\end{lemma}

\proof 
Writing $K=\colim E$ as the colimit of finitely generated field
extensions, we can assume (by the first property)
that $K$ is finitely generated. 
Then we can write it as a colimit
$K=\colim B$ of localizations of a finitely generated smooth $k$
algebra with function field $K$.
Consider the map $\f(k)\to \f(K)\cong\colim \f(B)$. 
If $x$ lies in the kernel,
then $x$ vanishes in $\f(B)$ for some $B$.
But $k\to B$ has a section by Hilbert's Nullstellensatz 
since $k$ is algebraically closed.
\proofend

It is clear by the $5$-Lemma that if the cohomology of two 
complexes satisfies
the hypothesis of the lemma, then also the cohomology
of the cone of a homomorphism between them. 
Sheaves represented by a scheme satisfy the hypothesis of the 
Lemma:

\begin{example}
Let $X$ be locally of finite presentation over $k$,
then for any inverse system of quasi-compact and quasi-separated schemes 
$(T_i)$ with affine transition morphisms we have 
\cite[Prop. 8.13.1]{EGAIV28}
$$\Hom_k(\lim T_i,X)\cong \colim \Hom_k(T_i,X).$$
In particular, for any scheme $Y$ of finite type over $k$, 
the functor $A\mapsto \Hom(Y_A,X)$ satisfies the hypothesis of 
Lemma \ref{inject}.
\end{example}

The following proposition shows
that Bloch's cycle complex $z^n(X,*)$ satisfies the hypothesis of 
Lemma \ref{inject}.

\begin{proposition}\label{mccolim}
a) Let $X/k$ be equidimensional of finite type and $k\subseteq K$ 
a field extension. Then we have an isomorphism  
$$\colim_{k\subset E\subset K,f.g.} z^n(X_E,*)\stackrel{\sim}{\to}
z^n(X_K,*)$$
as well as an isomorphism of complexes of \'etale sheaves on $X$
$$\colim_{k\subset E\subset K,f.g.} z^n(-,*)_{X_E}\stackrel{\sim}{\to}z^n(-,*)_{X_K}.$$

b) Let $Y\to S$ be an equidimensional morphism of $k$-schemes of finite 
type with $S$ integral. Then we have an isomorphism
$$\colim_{\emptyset\not=V\subseteq S \text{open}} z^n(Y\times_SV,*)
\stackrel{\sim}{\to}z^n(Y\times_Sk(S),*)$$
as well as an isomorphism of complexes of \'etale sheaves on $Y$ 
$$\colim_{\emptyset\not=V\subseteq S \text{open}} 
z^n(-,*)_{Y_V}\stackrel{\sim}{\to}z^n(-,*)_{Y_{k(S)}}.$$
\end{proposition}

\proof 
We prove the isomorphism on the presheaves level. Then the first
statements follow by taking global sections and the second by
sheafification. In both cases, the map is injective by 
\cite[Cor. 8.8.2.5]{EGAIV28} and we show the surjectivity.

a) Let $U\to X_K$ be \'etale. By \cite[Thm.8.8.2(ii)]{EGAIV28},
$U$ is the base change from $U_E\to X_E$ for some $E$, and by 
\cite[Thm.17.7.8(ii)]{EGAIV32} we can, after possibly increasing $E$, assume
that $U_E\to X_E$ is \'etale. Then every algebraic cycle in
$z^n(U,i)$ is the base change of a cycle from $z^n(U_E,i)$ by
\cite[Prop. 8.6.3, Cor. 8.7.3 b)]{EGAIV28}. Hence the statement
follows from the compatibility of 
the intersection with flat base-change.

b) The same argument works in this situation, except that this
time we use \cite[Prop. 8.6.3, Cor. 8.7.3 a)]{EGAIV28} to descend
the algebraic cycles.
\proofend

\begin{corollary}
Let $X/k$ be smooth. Then the functors
$A\mapsto H^i_\M(X_A,\Z/m(n))$ and $A\mapsto H^i_\et(X_A,\Z/m(n))$, and
the same with integral coefficients, satisfy the hypothesis of
Lemma \ref{inject}; hence the injectivity statement holds for them.
\end{corollary}

\proof
Since $X$ is smooth, so is $X_A$ for $A$ essentially smooth, and both
functors are functorial for all maps of $k$-algebras by pull-back. For
motivic cohomology, the proposition implies that the colimit hypotheses
of Lemma \ref{inject} are satisfied (we use b) with $Y=X\times_kS$).
For \'etale motivic cohomology, Proposition \ref{mccolim} shows that the
hypothesis of \cite[VII Thm. 5.7]{SGAIV2}, stating that \'etale
cohomology commutes with limits of schemes with affine transition
maps, is satisfied.
\proofend

\subsection*{Rigidity}
\label{app:suslin}
We verify that Suslin's argument \cite{suslin} goes through for Bloch's
higher Chow groups $\CH^n(Y,j;\Z/m)=H_j\big(z^n(Y,*)\otimes\Z/m\big)$
of smooth schemes, for every $m\geq 1$, and then deduce rigidity with
finite coefficients for arbitrary schemes of finite type by localization.
In this subsection, $k$ is algebraically closed, and until
Corollary \ref{app:sing}, $X/k$ is smooth and quasi-projective.
Besides Proposition \ref{mccolim}, we use the standard properties of
higher Chow groups of smooth quasi-projective schemes over a field
\cite{bloch-cycles, bloch-moving}: homotopy invariance, proper
push-forward, pull-back along morphisms of smooth schemes, the
localization sequence, and the product with $\CH^1(F,1)\cong F^\times$
for fields $F$ (and hence with products of units in $\CH^*(F,*)$), which
satisfies the projection formula and the Leibniz rule for the
localization boundary.
The coefficients enter the argument through the following fact:

\begin{itemize}
\item[$(\ast)$] {\it The groups $k^\times$ and $\Pic^0(C)$, for a smooth
projective curve $C/k$, are $m$-divisible.}
\end{itemize}

\noindent
This holds for all $m$ because multiplication
by $m$ is surjective on the $k$-points of $\G_m$ and of the Jacobian of
$C$. 
In particular, the product of any class in
$\CH^{n}(X,j;\Z/m)$ with $\{b\}\in k^\times\cong \CH^1(k,1)$ vanishes.

\subsubsection*{Specialization}
Let $C/k$ be a smooth connected curve with function field $L$, and
$c\in C$ a closed point with local ring $\O$, valuation $v$ and
uniformizer $\pi$. Bloch's localization sequence 
\cite{bloch-cycles, bloch-moving}
for $X\times c\subseteq X\times U$, in the colimit over the open
neighbourhoods $U$ of $c$ (Proposition \ref{mccolim} b)), gives a
boundary map
$$\partial_c:\CH^{n+1}(X_L,j+1;\Z/m)\to\CH^{n}(X,j;\Z/m),$$ 
and we define the specialization at $c$ to be the composition
$$s_c:\CH^n(X_L,j;\Z/m)\stackrel{\cdot \pi}{\longrightarrow}
\CH^{n+1}(X_L,j+1;\Z/m)
\stackrel{\partial_c}{\longrightarrow} \CH^n(X,j;\Z/m),$$
sending $t$ to $\partial_c\big(t\cdot\{\pi\}\big)$.
For $a\in\O^\times$ with residue $\bar a$, the Leibniz rule gives
$\partial_c(t\cdot\{a\})=\partial_c(t)\cdot\{\bar a\}$
(cf.\ \cite[(1.1)]{suslin}), which vanishes by $(\ast)$.
It follows that $s_c$ is independent of the uniformizer $\pi$, and
\begin{equation}\label{app:val}
\partial_c(t\cdot\{a\})=v(a)\,s_c(t)\qquad\text{for all } a\in L^\times.
\end{equation}
If $t=u|_L$ for a class $u$ defined on some $X\times U$, with
restriction $\bar u$ to $X\times c=X$, the Leibniz rule gives
$\partial_c(u|_L\cdot\alpha)=(-1)^j\,\bar u\cdot\partial_c(\alpha)$ for
$\alpha\in \CH^*(L,*)$ a product of units
(cf.\ \cite[(1.2)]{suslin}); for $\alpha=\{\pi\}$ this becomes
\begin{equation}\label{app:ext}
s_c(u|_L)=(-1)^j\,\bar u .
\end{equation}
These are the analogues of \cite[(1.3), (1.4)]{suslin}.

\begin{lemma}\label{app:rigidity}(cf.\ \cite[Rigidity Theorem]{suslin})
For closed points $c_0,c_1\in C$ we have $s_{c_0}=s_{c_1}$.
\end{lemma}

\proof
{\it Norms.}
Let $f:C'\to C$ be a finite morphism of smooth connected
curves, $L'=k(C')$, and $N$ the push-forward along the finite morphism
$X_{L'}\to X_L$. Let $c'_1,\dots,c'_r$ be the points over $c$, and
$e_\nu=v_{c'_\nu}(\pi)$ the ramification indices. 
Since $f$ is finite, hence closed, the sets $f^{-1}(U)$ for $U$ a
neighbourhood of $c$ are cofinal among the neighbourhoods of
$\{c'_1,\dots,c'_r\}$. The boundary map for
$X\times f^{-1}(c)=\coprod_\nu X\times c'_\nu\subseteq X\times f^{-1}(U)$
has components $\partial_{c'_\nu}$, as one sees by restricting to the open
subsets $f^{-1}(U)\smallsetminus\{c'_\mu:\mu\neq\nu\}$. Push-forward
along the proper map $X\times f$ gives a map of localization sequences,
and since $k$ is algebraically closed, $X\times c'_\nu\to X\times c$ is
the identity of $X$. Hence (cf.\ \cite[Lemma 1]{suslin})
\begin{equation}\label{app:bdnorm}
\partial_c\circ N=\textstyle\sum_\nu\partial_{c'_\nu}.
\end{equation}
The projection formula for $X_{L'}\to X_L$ says that
$N(t\cdot f^*\beta)=N(t)\cdot\beta$ for $t\in\CH^n(X_{L'},j;\Z/m)$
and $\beta\in\CH^1(L,1)=L^\times$. 
Hence
$$s_c(N(t))=\partial_c\big(N(t)\cdot\{\pi\}\big)
=\partial_c\big(N(t\cdot\{\pi\})\big)
=\sum_\nu\partial_{c'_\nu}\big(t\cdot\{\pi\}\big)
=\sum_\nu v_{c'_\nu}(\pi)\,s_{c'_\nu}(t).$$
The first equality is the definition of $s_c$, the second is the
projection formula, the third is \eqref{app:bdnorm}, and the last is
\eqref{app:val} at $c'_\nu$, applied to $a=\pi\in L'^\times$. 
Thus (cf.\ \cite[Lemma 3]{suslin}) 
\begin{equation}\label{app:norm}
s_c\circ N=\textstyle\sum_\nu e_\nu\, s_{c'_\nu}.
\end{equation}

{\it The projective line.} Let $L=k(T)$. For a class $u$ on $X$ we write
$u_L$ for its pull-back along the projection $X_L\to X$; it is the
restriction of the pull-back of $u$ to $X\times\A^1$, whose restriction
to $X\times c$ is $u$ for every $c$, so \eqref{app:ext} gives
$s_c(u_L)=(-1)^j u$ for $u\in\CH^n(X,j;\Z/m)$. Localization at the
points of $\A^1(k)$ and homotopy invariance give a long exact sequence
$$\CH^n(X,j;\Z/m)\xrightarrow{u\mapsto u_L}\CH^n(X_L,j;\Z/m)
\stackrel{(\partial_a)}{\longrightarrow}
\bigoplus_{a\in\A^1(k)}\CH^{n-1}(X,j-1;\Z/m).$$
For $u\in\CH^{n-1}(X,j-1;\Z/m)$ and $a,b\in\A^1(k)$,
\eqref{app:val} and \eqref{app:ext} give
$$\partial_a\big(u_L\cdot\{T-b\}\big)=v_a(T-b)\,s_a(u_L)
=(-1)^{j-1}\delta_{ab}\,u .$$
So $(u_a)_a\mapsto(-1)^{j-1}\sum_a (u_a)_L\cdot\{T-a\}$ splits 
$(\partial_a)$,
and $\CH^n(X_L,j;\Z/m)$ is generated by the classes $u_L$ with
$u\in\CH^n(X,j;\Z/m)$ and the classes $u_L\cdot\{T-a\}$ with
$u\in\CH^{n-1}(X,j-1;\Z/m)$. Let $c\in\A^1(k)$. On the former,
$s_c(u_L)=(-1)^ju$ by \eqref{app:ext}. On the latter, by the Leibniz rule
(cf.\ \cite[Lemma 2]{suslin}),
$$s_c\big(u_L\cdot\{T-a\}\big)
=\partial_c\big(u_L\cdot\{T-a\}\cdot\{\pi\}\big)
=(-1)^{j-1}\,u\cdot\partial_c\big(\{T-a\}\cdot\{\pi\}\big)=0,$$
by $(\ast)$, since $\partial_c(\{T-a\}\cdot\{\pi\})\in\CH^1(k,1)$.
Thus $s_c$ is independent of $c\in\A^1(k)$. Since any two points of
$\P^1(k)$ lie in the complement of a third, $s_c$ is independent of
$c\in\P^1(k)$.

{\it Arbitrary curves.} Since $s_c$ only depends on $\O_{C,c}$, we can
replace $C$ by its smooth compactification and assume that $C$ is
projective. Consider the homomorphism
$$q:\CH^n(X_L,j;\Z/m)\otimes\Div(C)\to\CH^n(X,j;\Z/m),\qquad
t\otimes\textstyle\sum n_\nu c_\nu\mapsto\sum n_\nu\, s_{c_\nu}(t).$$
Let $f\in L^\times$ be nonconstant. Then $f:C\to\P^1$ is finite, and we
apply \eqref{app:norm} to it, with $L\supseteq k(T)$ in place of
$L'\supseteq L$, and with $N$ the push-forward along
$X_L\to X_{k(T)}$. At $0$ we take the uniformizer $T$, whose pull-back
to $C$ is $f$, so that the ramification index at a zero $c'$ of $f$ is
$v_{c'}(f)$. At $\infty$ we take the uniformizer $T^{-1}$, so that the
ramification index at a pole $c'$ of $f$ is $-v_{c'}(f)$. Hence
$$s_0(N(t))-s_\infty(N(t))
=\sum_{c'}v_{c'}(f)\,s_{c'}(t)=q\big(t\otimes\div(f)\big),$$
and the left-hand side vanishes by the case of $\P^1$, applied to
$N(t)\in\CH^n(X_{k(T)},j;\Z/m)$. Thus $q$ vanishes on
$\CH^n(X_L,j;\Z/m)\otimes\operatorname{Prin}(C)$. 
The exact sequence
$\operatorname{Prin}(C)\to\Div^0(C)\to\Pic^0(C)\to0$ then shows that the
restriction of $q$ to $\CH^n(X_L,j;\Z/m)\otimes\Div^0(C)$ factors
through $\CH^n(X_L,j;\Z/m)\otimes\Pic^0(C)=0$ 
(because the first factor is killed by $m$ and the second is
$m$-divisible by $(\ast)$).
Thus $q$ vanishes on $\CH^n(X_L,j;\Z/m)\otimes\Div^0(C)$, and applied to
$c_0-c_1\in\Div^0(C)$, this gives
$s_{c_0}(t)-s_{c_1}(t)=q\big(t\otimes(c_0-c_1)\big)=0$.
\proofend

\begin{theorem}\label{app:main}
Let $k\subseteq\K$ be an extension of algebraically closed fields and
$X/k$ smooth and quasi-projective. Then for all $n,j$ and $m\geq 1$,
the maps
\begin{align*}
\CH^n(X,j;\Z/m)&\to\CH^n(X_\K,j;\Z/m),\\
\CH^n(X,j)/m&\to\CH^n(X_\K,j)/m,\\
\CH^n(X,j)[m]&\to\CH^n(X_\K,j)[m]
\end{align*}
are isomorphisms.
\end{theorem}

\proof
Write $R(A)=\CH^n(X_A,j;\Z/m)$. If $A$ is a smooth connected
$k$-algebra of finite type and $h_0,h_1:A\to k$, then
$h_0^*=h_1^*:R(A)\to R(k)$: the two points lie on an irreducible curve
in $\Spec A$ \cite[p.\ 56]{mumford}; if $g:C\to\Spec A$ is its
normalization and $c_\nu$ lies over $h_\nu$, then
$h_\nu^*(x)=(-1)^j s_{c_\nu}(g^*x|_L)$ by \eqref{app:ext}, and
Lemma \ref{app:rigidity} applies. The same holds with $k$ replaced by
$\K$ and $X$ by $X_\K$.

Injectivity of $R(k)\to R(\K)$ is Lemma \ref{inject}: $R$ is functorial
since $X$ is smooth, and the colimit hypotheses hold by
Proposition \ref{mccolim}. For surjectivity, let $x\in R(\K)$.
By Proposition \ref{mccolim}, and since $k$ is perfect, $x$ is the
image of some $x_A\in R(A)$ for a smooth finitely generated
$k$-subalgebra $A\subseteq\K$, which is connected since
$A\subseteq\K$. Choose $h:A\to k$. The $\K$-points
$a\otimes\lambda\mapsto a\lambda$ and $a\otimes\lambda\mapsto h(a)\lambda$
of the smooth connected $\K$-variety $\Spec(A\otimes_k\K)$ induce the same
map $R(A\otimes_k\K)\to R(\K)$. Evaluating on the image of $x_A$ shows
that $x$ is the image of $h^*(x_A)\in R(k)$.

For the remaining statements, consider the map of universal coefficient
sequences
$$0\to\CH^n(-,j)/m\to\CH^n(-,j;\Z/m)\to\CH^n(-,j-1)[m]\to 0$$
for $X\to X_\K$. The middle map is an isomorphism, so the right map is
surjective. The functors $A\mapsto\CH^n(X_A,j)/m$ and
$A\mapsto\CH^n(X_A,j-1)[m]$ are functorial since $X$ is smooth, and
commute with filtered colimits by Proposition \ref{mccolim}, so the
outer maps are injective by Lemma \ref{inject}.
Hence the right map is an isomorphism, and so is
the left by the five lemma.
\proofend

\subsubsection*{Arbitrary schemes}
For $Y$ of finite type over a field, let $z_d(Y,*)$ be the version of
$z^n(Y,*)$ indexed by dimension, with $z_d(Y,j)$ generated by cycles of
dimension $d+j$ on $Y\times\Delta^j$, and $\CH_d(Y,j;\Z/m)$ its homology
with $\Z/m$-coefficients. Thus $\CH_d(Y,j;\Z/m)=\CH^{e-d}(Y,j;\Z/m)$
if $Y$ is equidimensional of dimension $e$. For a closed subscheme
$Z\subseteq Y$ with open complement $U$, there is a long exact
localization sequence
$$\cdots\to\CH_d(Z,j;\Z/m)\to\CH_d(Y,j;\Z/m)\to\CH_d(U,j;\Z/m)
\to\CH_d(Z,j-1;\Z/m)\to\cdots,$$
see \cite{bloch-moving} for $Y$ quasi-projective and
\cite{levine-localization} for $Y$ separated of finite type. It is
compatible with flat pull-back, in particular with extension of the
base field.

\begin{corollary}\label{app:sing}
Let $k\subseteq\K$ be an extension of algebraically closed fields and
$X/k$ separated of finite type. Then for all $d,j$ and $m\geq 1$, the map
$$\CH_d(X,j;\Z/m)\to\CH_d(X_\K,j;\Z/m)$$
is an isomorphism. 
\end{corollary}

\proof
We use induction on $\dim X$. Since $z_d(X,*)=z_d(X_{\rm red},*)$, we
can assume that $X$ is reduced. Since $k$ is perfect, the smooth locus
of $X$ is dense, so there is an open affine subscheme $U\subseteq X$
which is smooth and dense in the union of the irreducible components of
maximal dimension: take in each such component an affine open subscheme
which is contained in the smooth locus and disjoint from the other
components. Let $Z=X\smallsetminus U$, so that $\dim Z<\dim X$, and
$Z=\emptyset$ if $\dim X=0$. Flat pull-back along $X_\K\to X$ gives a
map from the localization sequence for $Z\subseteq X$ to the
localization sequence for $Z_\K\subseteq X_\K$. Since $U$ is smooth,
affine and equidimensional, the maps for $U$ are isomorphisms by
Theorem \ref{app:main}, and the maps for $Z$ are isomorphisms by
induction. Hence the maps for $X$ are isomorphisms by the five lemma.
\proofend

The proof does not give the statements for $\CH_d(X,j)/m$ and
$\CH_d(X,j)[m]$ for singular $X$: Proposition \ref{mccolim} gives
the colimit hypotheses of Lemma \ref{inject}, but
$A\mapsto\CH_d(X_A,j)$ is not known to be functorial for the sections
$B\to k$ used in the proof of Lemma \ref{inject}.

\section{Lang-Steinberg and logarithmic de Rham-Witt sheaves}
\subsection*{Frobenius action}
For an $\F_q$-scheme $S$, the absolute Frobenius $\varphi_S$
is the map which is the identity on the topological space of $S$ and 
the $q$th power map $x\to x^q$ on the structure sheaf. 
It is easy to see that for a map of
schemes $f:S\to S'$ we have $f\circ \varphi_S=\varphi_{S'}\circ f$. 
If $S$ is regular noetherian, then $\varphi_S$ is flat by Kunz's theorem,
and if $S$ is of finite type over a perfect $\F_p$-algebra, then $\varphi_S$
is finite. 
For an $\F_q$-scheme $X$, and an $\F_q$-scheme $S$ (usually the spectrum 
of an algebraically closed field or a perfect $\F_q$-algebra), 
the "arithmetic Frobenius" (relative to $S$) is the map
$X\times_{\F_q}S\stackrel{\id\times\varphi_S}{\longrightarrow}
X\times_{\F_q}S$.  For a presheaf $\f$ on a suitable subcategory of 
$\F_q$-schemes, this induces the Frobenius map 
$F:\f(X\times_{\F_q}S)\to \f(X\times_{\F_q}S)$, which is a
morphism of presheaves compatible with morphisms $S\to S'$.

This applies, for example, to contravariant functors like motivic
cohomology and \'etale cohomology, i.e., we act on algebraic cycles via
pull-back along the arithmetic Frobenius map.

\subsection*{The Lang-Steinberg theorem}
For a group scheme $G$ over $\F_q$, the absolute Frobenius
induces a map of $S$-valued points $G(S)$ sending $s:S\to G$ to
$\varphi_G\circ s$. 
This agrees with the above action $F$ on the presheaf represented by $G$
because $\varphi_G\circ s=s\circ \varphi_S$.

\begin{theorem}(Lang-Steinberg) For any connected
algebraic group $G$ of finite type over $\F_q$, 
the map $\varphi_G-1$ is surjective with finite kernel.
\end{theorem}

\proof 
The surjectivity of $\varphi_G-1$ is proven in 
\cite[Thm. 4.4.17]{springer} and \cite{lang}, and
the kernel $H$ is a zero-dimensional scheme of finite type,
hence a finite group scheme over $\F_q$.
\proofend

\begin{example} We have short exact sequences of abelian groups
\begin{equation}\label{gacase}
0\to \F_q\to \K \stackrel{\varphi-1}{\longrightarrow} \K\to 0, 
\end{equation}
\begin{equation}\label{gmcase}
0\to \F_q^\times\to \K^\times \stackrel{\varphi-1}{\longrightarrow} 
\K^\times\to 0
\footnote{The map is $(\varphi-1)(u)=\frac{\varphi(u)}u$, so that 
$(\varphi-1)(u)=1\Leftrightarrow \varphi(u)=u$.}.
\end{equation}
For abelian varieties $A$, the map 
$$A(\K) \stackrel{F-1}{\longrightarrow} A(\K)$$
is surjective with finite kernel independent of $\K$.
\end{example}

\begin{corollary}\label{locftlang}
For any commutative group scheme $G$ locally of finite type over $\F_q$, 
the kernel and cokernel of $F-1$ on $G(\K)$ do not depend on the 
algebraically closed field $\K$ containing $\F_q$. 
\end{corollary}

\proof
By \cite[II \S 5 Prop. 1.8]{DG} we have a short exact sequence
$$ 0\to G^0 \to G\to \pi_0(G)\to 0$$
with $G^0$ connected and of finite type, and $\pi_0(G)$ \'etale. 
Because $\pi_0(G)$ is an \'etale group scheme, its $\K$-rational points
are canonically identified with its $\bar \F_q$-rational points, 
thus $F-1$ acts on a discrete group independent of $\K$.
Taking $\K$-rational points is exact, so that it suffices to show that
$G^0(\K)\stackrel{F-1}{\longrightarrow}G^0(\K)$ is surjective with kernel
independent of $\K$. But a surjection of varieties of
finite type induces a surjection of $\K$-rational points
for any algebraically closed field $\K$. As for the finite kernel,
if $H=\Spec A$, then
$H(\K)=H(\K')$ for any extension of algebraically closed fields
because any morphism $A\to \K'$ factors through $\K$ as all
residue fields of $A$ are finite extensions of $\F_q$.
\proofend

\subsection*{Cohomology of logarithmic de Rham-Witt sheaves}
We recall Milne's results on the \'etale cohomology of
the logarithmic de Rham-Witt sheaves.
Let $k$ be a perfect field of characteristic $p$
and $\mathrm{Pf}/k$ be the category of perfect affine 
schemes over $k$, i.e., the spectra of $k$-algebras $A$ satisfying $A=A^p$.
A group object $G\in \mathrm{Pf}/k$ is called a perfect group scheme, and
is called "algebraic" if there
exists an affine group scheme $G_0$ of finite type over $k$ 
such that $G(A)=G_0(A)$ 
for all perfect $k$-algebras $A$. Then $G$ is called the perfection
$G_0^{pf}$ of $G_0$. 
Let $\GG$ be  the category of commutative algebraic perfect group schemes, 
which are killed by some power of $p$. If we equip $\mathrm{Pf}/k$ with 
the \'etale topology, then $\GG$ is a full abelian subcategory of the 
category of sheaves on $(\mathrm{Pf}/k)_\et$ closed under extensions.

\begin{lemma}\cite[Lemma 1.8]{milne_values}
If $X$ is smooth and proper, then
the functor $T\mapsto H^i_\et(X_T,\nu_r(n))$ on $(\mathrm{Pf}/k)_\et$ 
is represented by an object of $\GG$. 
\end{lemma}


We write $\underline H^i_\et(X,\nu_r(n))$ for the representing object,
$U^i(X,\nu_r(n))$ for its connected component, and 
$D^i(X,\nu_r(n))=\underline H^i_\et(X,\nu_r(n))/ U^i(X,\nu_r(n))$
for its \'etale quotient. 

\begin{theorem}\cite[Thm. 1.11]{milne_values}
If $X$ is smooth and proper of dimension $d$ over a perfect field $k$,
then
\begin{align*}
 U^i(X,\nu_r(n))&\cong \Ext^1(U^{d+1-i}(X,\nu_r(d-n)),\Q_p/\Z_p)\\
D^i(X,\nu_r(n))&\cong \Hom(D^{d-i}(X,\nu_r(d-n)),\Q_p/\Z_p).
\end{align*}
\end{theorem}

\begin{proposition}\label{n=d}
We have $U^i(X,\nu_r(0))=U^i(X,\nu_r(d))=0$ for all $i$ and $r$.
\end{proposition}

\proof
The statement for $n=0$ follows from the proper base change theorem,
and then the statement for $n=d$ follows by duality.
\proofend

The identity component $G^0$ of an object $G$ of  $\GG$ has a composition
series whose quotients are isomorphic to $\G_a^{pf}$, and the quotient
$G_\et$ is \'etale \cite[Lemme 2.4]{berth}. 

\begin{corollary}\label{GGobjects}
If $k$ is a finite field, then for an object $G$ of $\GG$, the kernel 
and cokernel of $F-1$ 
on $G(\K)$ are independent of the algebraically closed field
$\K$ and are finite.
\end{corollary} 

\proof
Given an extension $\K'/\K$ of algebraically closed fields,
compare the following diagram
$$\begin{CD}
0@>>> G^0(\K)@>>> G(\K)@>>> G_\et(\K) @>>> 0\\
@. @VF-1VV@VF-1VV@VF-1VV\\
0@>>> G^0(\K)@>>> G(\K)@>>> G_\et(\K) @>>> 0
\end{CD}$$
with the corresponding diagram for $\K'$. 
Filtering $G^0(\K)$ with quotients $\G_a^{pf}$ and using the sequence
\eqref{gacase} we see that 
the left vertical map is surjective with kernel independent of $\K$
by the Lang-Steinberg theorem,
and the two right terms are the same for $\K'$.
\qed

\section{Motivic cohomology over algebraically closed fields}
In this section we discuss how motivic cohomology of a smooth
and proper scheme over an algebraically closed field $\K$ 
changes if we base extend to another algebraically closed field.
Since
$$H^1_\M(\K,\Z(1))\cong H^1_\et(\K,\Z(1))\cong \K^\times,$$
depends on $\K$, and since every variety $X$ over $\K$ has a $\K$-rational
point, so that $H^1_\M(\K,\Z(1))$ is a direct summand of 
$H^1_\M(X,\Z(1))$, 
the question is only meaningful with finite coefficients.

\begin{theorem}\label{fincoef}
1) [Suslin rigidity] For any $m\geq 1$, the groups 
$H^i_\M(X_\K,\Z/m(n))$ do not depend on $\K$.

2) If $\frac{1}{m}\in \K$, then the groups $H^i_\et(X_\K,\Z/m(n))$ 
are finite and do not depend on $\K$.
\end{theorem}

\proof 
1) This follows from Corollary \ref{app:sing}, applied to the connected
components of $X$, since $H^i_\M(X,\Z/m(n))\cong\CH^n(X,2n-i;\Z/m)$
for $X$ smooth, and $\CH^n=\CH_{e-n}$ on a smooth connected scheme of
dimension $e$.

2) This follows from $\Z/m(n)\cong \mu_m^{\otimes n}$ by 
\cite{geisser-levine}, and the proper base-change theorem
\cite[Cor. VI.2.6]{milne-etale}.
\proofend

In contrast, it follows from Theorem \ref{schoen} that 
$H^i_\M(X_\K,\Z/m(n))$ can be infinite (in characteristic $0$),
but we have the following finiteness result for certain weights:

\begin{proposition}
For any $m\geq 1$ and $n=0, 1,d$, the groups 
$H^i_\M(X_\K,\Z/m(n))$ are finite.
\end{proposition}

\proof
For $n=0$ the non-vanishing group is
$H^0_\M(X_\K,\Z/m)\cong (\Z/m\Z)^{\pi_0(X_\K)}$,
which does not depend on $\K$.

For $n=1$, since $\K^\times$ and $\Pic^0(X_\K)$ are divisible, it follows
from the coefficient sequence that for any $m$, 
$H^1_\M(X_\K,\Z/m(1))\cong \Pic(X_\K)[m]$
and $H^2_\M(X_\K,\Z/m(1))\cong \NS(X_\K)/m$. 
These groups do not depend on $\K$.

Finally, for $n=d$ the Zariski and \'etale cohomology agree 
\cite[Thm. 3.1]{ichannals}, hence the statement follows from 
Theorem \ref{fincoef}(2) and
Proposition \ref{n=detale}.
\proofend

The $p$-part of the \'etale theory does depend on $\K$.
For example, for a supersingular abelian surface or K3-surface, 
$$H^2_\et(X_\K,\Z/p(1))\twoheadrightarrow \Br(X_\K)\cong \K,$$
and the kernel $\NS(X_\K)/p$ does not depend on $\K$. Again we have
some stronger results in certain weights:

\begin{proposition}\label{n=detale}
For $n=0, d$, the groups 
$H^i_\et(X_\K,\Z/p(n))$ do not depend on $\K$, and are finite.
\end{proposition}

\proof
For $n=0$ this is the proper base change theorem. For $n=d$, 
it follows from Proposition \ref{n=d} that 
$\underline H^i_\et(X_\K,\Z/p(n))$
is the perfection of a finite \'etale group scheme.
\proofend

Properness is essential because Artin-Schreier
theory shows that $H^1_\et(\A^1_\K,\Z/p(0))$ is an infinite
dimensional $\F_p$-vector space.
We pose the remaining cases as a question:

\begin{question}\label{Q1}
Let $X/k$ be smooth and proper over an algebraically
closed field of characteristic $p$. Are the groups
$H^i_\M(X_\K,\Z/m(n))$ finite for $p\nmid m$ or for $m=p$?
\end{question}

\subsection*{Chow groups of threefolds}
The first groups where we do not know finiteness,
and which are not covered by the above results, are
$\CH^2(X)[m]$ and $\CH^2(X)/m$ for threefolds. 
For any $m$, we have an exact sequence 
\begin{multline}\label{dece}
0\to H^3_\M(X_\K,\Z/m(2))\to H^3_\et(X_\K,\Z/m(2))
\to  H^0_\Zar(X_\K,R^3\epsilon_* \Z/m(2))\\
\to H^4_\M(X_\K,\Z/m(2))\to H^4_\et(X_\K,\Z/m(2))\to\cdots
\end{multline}
for $\epsilon: X_\et \to X_\Zar$ the change of topology morphism.
Away from the characteristic, $H^3_\et(X_\K,\Z/m(2))$ is finite
and does not depend on $\K$ by Theorem \ref{fincoef}, 
hence \eqref{dece} together with Lemma \ref{inject} shows
that the same is true for $H^3_\M(X_\K,\Z/m(2))$. Then the 
coefficient sequence 
$$0\to H^3_\M(X_\K,\Z(2))/m \to  H^3_\M(X_\K,\Z/m(2))
\to \CH^2(X)[m]\to 0$$
and Lemma \ref{inject} show that $\CH^2(X_\K)[m]$ is also 
finite and does not depend on $\K$.
Schoen constructed examples where $\CH^2(X)/m$
is infinite in characteristic $0$
\footnote{No such example is known in characteristic $p$.}:

\begin{theorem}\cite[Thm. 0.2]{schoen}\label{schoen}
Let $\K$ be an algebraically closed field of characteristic $p\not=3$
and let $E\subset {\mathbb P}^2_\Z$ be defined by the equation 
$x_0^3+x^3_1+ x^3_2= 0$. Let $l$ be a prime number. If $p>0$,
then $\CH^2(E^3_\K)/l$ is finite. If $p=0$
and $l \equiv 1 \bmod 3$, then $\CH^2(E^3_\K)/l$ is infinite.
\end{theorem}

The sequence \eqref{dece} shows that this implies that
$H^0_\Zar(X_\K,R^3\epsilon_* \Z/l(2))$ is then also infinite.

\begin{remark}
Schoen's example implies that there is no analogue of the proper
base-change theorem for motivic cohomology: 

Let $l\equiv 1\bmod 3$ and $p\not=3,l$, and let $V$ be the 
strict henselization
of $\Z$ at $p$ with residue field $k=\bar \F_p$ and quotient field $K$.
For any finite extension field $K'$  of $K$, the normalization $V'$
of $V$ in $K'$ is another strict henselian discrete valuation ring. 
We claim that the pull-back map to the special fiber
$H^4_\M(E^3_{V'},\Z/l(2))\to H^4_\M(E^3_k,\Z/l(2))$
can not be an isomorphism for all $K'$. Indeed,
if this were the case, then in the colimit over $K'$, 
$$\CH^2(E^3_k)/l\stackrel{\sim}{\to}
H^4_\M(E^3_k,\Z/l(2))\stackrel{\sim}{\leftarrow}\colim_{K'}H^4_\M(E^3_{V'},\Z/l(2))$$ 
surjects onto
$$\colim_{K'}H^4_\M(E^3_{K'},\Z/l(2))\cong 
H^4_\M(E^3_{\bar \Q},\Z/l(2))\cong \CH^2(E^3_{\bar \Q})/l.$$
However the former group is finite and the latter group is infinite 
by Schoen's theorem.
\end{remark}

\section{The fiber of Frobenius}
Given a smooth and proper variety over a finite field $\F_q$ with 
algebraic closure $\F$, we can ask whether the fiber of $F-1$
for $F$ the Frobenius 
is unchanged if we base change to a larger algebraically closed field $\K$. 
We redefine Weil-\'etale cohomology with coefficients in $A$ to allow
the extra parameter:
$$H^i_W(X_\K,A(n))= H^i\Bigl(\fiber F-1: 
R\Gamma_\et(X_\K,A(n))\to R\Gamma_\et(X_\K,A(n))\Bigr).$$
Similarly we consider the fiber of $F-1$ on motivic cohomology:
$$H^i_F(X_\K,\Z(n))= H^i\Bigl(\fiber F-1: 
R\Gamma_\M(X_\K,\Z(n))\to R\Gamma_\M(X_\K,\Z(n))\Bigr).$$
These groups and their interrelations have been studied in \cite{ichfrob}.
By construction we have short exact sequences
$$ 0\to H^{i-1}_\et(X_\K,A(n))_{F-1}\to H^i_W(X_\K,A(n))\to H^i_\et(X_\K,A(n))^{F-1}\to 0$$
and similarly for $H^i_F(X_\K,\Z(n))$.
In particular, we have short exact sequences
$$ 0\to H^{2n-1}_\M(X_\K,\Z(n))_{F-1}\to H^{2n}_F(X_\K,\Z(n))
\to CH^n(X_\K)^{F-1}\to 0, $$
an isomorphism
$$\CH^n(X_\K)_{F-1}\cong H^{2n+1}_F(X_\K,\Z(n)),$$
and $H^i_F(X_\K,\Z(n))$ vanishes for $i>2n+1$.

\subsection*{Finite coefficients}
Our first result is that with torsion coefficients,
Weil-\'etale cohomology does not depend on $\K$.

\begin{theorem}\label{frobmodm}
For all $i,n\geq 0$, and $m\geq 1$, the groups $H^i_W(X_\K,\Z/m(n))$ 
are finite and do not depend on $\K$.
\end{theorem}

\proof
We can consider the case that $m$ is prime to $p$ and that
$m$ is a power of $p$ separately. 
If $p\nmid m$, then \'etale motivic cohomology is finite and independent 
of $\K$ even before taking the fiber of $F-1$ by Theorem \ref{fincoef}. 
For the case that $m$ is a power of $p$, 
the result follows from
Corollary \ref{GGobjects} because $\Z/p^r(n)\cong \nu_r(n)[-n]$. 
\qed

\begin{corollary}
For all $i>2n+1$, the groups $H^i_W(X_\K,\Z(n))$ do not depend on 
$\K$ and are of cofinite type.
\end{corollary}

\proof 
By the coefficient sequence $H^{i-1}_W(X_\K,\Q/\Z(n))$ surjects onto 
$H^{i}_W(X_\K,\Z(n))$ because the latter is torsion for $i>2n+1$. 
Comparing field extensions, surjectivity of
$H^{i}_W(X_{\F},\Z(n))\to H^{i}_W(X_\K,\Z(n))$ then follows from Theorem 
\ref{frobmodm}, and injectivity follows from Lemma \ref{inject}.
\proofend

\begin{corollary}
The Frobenius invariants $\Br(X_\K)^{F-1}$ and coinvariants $\Br(X_\K)_{F-1}$
of the Brauer group do not depend on $\K$.
\end{corollary}

\proof
Consider the self map $F-1$ on the short exact 
coefficient sequence
$$ 0\to \NS(X_\K)\otimes\Q/\Z\to H^2_\et(X_\K,\Q/\Z(1))\to \Br(X_\K)\to 0.$$
Since the kernel and cokernel of $F-1$ on the left and middle terms do not 
depend on $\K$, the Snake Lemma implies the same holds true for the Brauer group.
\proofend

For motivic cohomology, Suslin rigidity implies that 
$H^i_F(X_\K,\Z/m(n))$ does not depend on $\K$,  
and we have the following weaker version of Question \ref{Q1}:

\begin{question}
Let $X/\F_q$ be smooth and proper. Are the groups
$H^i_F(X_\K,\Z/m(n))$ finite for $p\nmid m$ or $m=p$?
\end{question}

\subsection*{Rational coefficients}
By Theorem \ref{frobmodm}, rigidity for Weil-\'etale cohomology 
reduces to rational coefficients. Moreover, motivic cohomology in
the Zariski and \'etale motivic topology agrees 
with rational coefficients, so that it suffices to 
consider one of them, say $H^i_F(X_\K,\Q(n))$.

\begin{question}\label{rigid?}
Let $X/\F_q$ be smooth and proper. Are the groups
$H^i_F(X_\K,\Q(n))$ independent of $\K$, and are they finite
dimensional?
\end{question}

Question \ref{rigid?} has an affirmative answer for $n=0,1$:

\begin{proposition}
For $n=0,1$, $ H^i_F(X_\K,\Z(n))$ does not depend on $\K$ and is of
finite rank.
\end{proposition}

This was proved by Richarz-Scholbach \cite{RS} after inverting $p$.

\proof
For $n=0$, the only non-zero motivic cohomology group is 
$H^0_\M(X_\K,\Z(0))\cong \Z^{\pi_0(X_\K)}$
and this does not depend on $\K$ (even before taking the fiber of
$F-1$).
For $n=1$, we have  $H^1_\M(X_\K,\Z(1))\cong \K^\times$, hence
the result follows from \eqref{gmcase}. The only other non-vanishing
group is $H^2_\M(X_\K,\Z(1))\cong \Pic(X_\K)$. These are the
$\K$-rational points of the commutative group scheme $\Pic_X$
locally of finite type over $\F_q$, hence the result
follows from  Corollary \ref{locftlang}.
\proofend

\subsection*{Connection to the Beilinson-Parshin conjecture}
Parshin's conjecture is equivalent to the vanishing of 
$H^i_\M(X,\Q(n))$ for $i\not=2n$ and smooth and projective
$X$ over $\F_q$ or, equivalently, over $\F$. 
This is clearly wrong over bigger fields, as the example 
$H^1_\M(\K,\Z(1))=\K^\times$ shows. We propose the following
version of Parshin's conjecture over arbitrary algebraically
closed fields using Frobenius rigidity.

\begin{proposition} The following statements are equivalent:

1) For any smooth and projective $X$ over $\F_q$, any 
algebraically closed field $\K$ containing $\F_q$ and any $n\geq 0$, 
the map 
$$ F-1: \CH_{\hom}^n(X_\K)_\Q\to \CH_{\hom}^n(X_\K)_\Q$$
as well as the maps
$$F-1: H^i_\M(X_\K,\Q(n))\to H^i_\M(X_\K,\Q(n))$$
for $i\not=2n$ are isomorphisms. 

2) Parshin's conjecture holds, Beilinson's conjecture
that rational and homological
equivalence agree up to torsion over finite fields holds, and 
the independence part of 
Question \ref{rigid?} has an affirmative answer. 
\end{proposition}

\begin{lemma}\label{finiteorbits}
Let $V$ be a $\mathbb{Q}$-vector space equipped with a linear
automorphism $F$. Assume that every $v\in V$ has a finite $F$-orbit.
Then the natural map $ V^{F-1} \rightarrow V_{F-1}$
is an isomorphism. 
\end{lemma}

The sequence 
$$ 0\to \bigoplus_\Z\Q \stackrel{F-1}{\longrightarrow} \bigoplus_\Z\Q
\stackrel{\sum}{\longrightarrow} \Q\to 0,$$
where $F$ acts like the shift, shows that the hypothesis is necessary.

\proof
For $v \in V$ with $F^n v = v$, set $\pi_n(v) = \frac1n\sum_{i=0}^{n-1} F^i v$. 
It is easy to check that if also $F^{dn}v=v$, then
$\pi_{dn}(v) =  \pi_n(v)$. We write $\pi(v)$ for this value and
claim that this gives an inverse to the natural map. 
First observe that $\pi(F(v)) = \pi(v)$ so that $\pi$ is well-defined on $V_{F-1}$
and its image lies in $V^{F-1}$. Furthermore $\pi|_{V^{F-1}} = \operatorname{id}$
since $Fv=v$ gives $\pi(v)=\tfrac1n\cdot nv = v$.
Conversely, $v - \pi(v) = \frac1n\sum_{i=0}^{n-1}(1-F^i)v \;\in\; 
\operatorname{im}(F-1)$,
since $1-F^i = -(F-1)(1+F+\cdots+F^{i-1})$, so that $v$ and $\pi(v)$
agree in $V_{F-1}$.
\proofend

By Proposition \ref{mccolim}, 
$H^i_\M(X_\F,\Q(n))\cong \colim_r H^i_\M(X_{\F_{q^r}},\Q(n))$ 
satisfies the hypothesis of the lemma. Since 
$H^i_\M(X,\Q(n))\cong H^i_\M(X_{\F_{q^r}},\Q(n))^{F-1}$ for all
$r$ by  a norm argument, this implies that  
$H^i_\M(X,\Q(n))=H^i_\M(X_\F,\Q(n))^{F-1}$.

\proof

$1) \Rightarrow 2)$: For $\K=\F $, we have 
$H^i_\M(X,\Q(n))=H^i_\M(X_\F,\Q(n))^{F-1}=0$ for $i<2n$, hence
Parshin's conjecture holds. The first statement implies that
$\CH_{\hom}^n(X)_\Q\cong \CH_{\hom}^n(X_{\F })_\Q^{F-1}$ vanishes,
hence Beilinson's conjecture. Finally, we  will see in Theorem 
\ref{chowrigid} that $\CH^n(X_\K)_\Q/\CH_{\hom}^n(X_\K)_\Q$ is
independent of $\K$, so that the diagram 
$$\begin{CD}
0@>>> \CH_{\hom}^n(X_{\K })_\Q@>>> \CH^n(X_{\K })_\Q
@>t>> \CH^n(X_{\K })_\Q/\CH_{\hom}^n(X_\K)_\Q@>>> 0\\
@.@VF-1VV@VF-1VV @VF-1VV\\
0@>>>\CH_{\hom}^n(X_{\K })_\Q@>>> \CH^n(X_{\K })_\Q
@>t>> \CH^n(X_{\K })_\Q/\CH_{\hom}^n(X_\K)_\Q@>>> 0\\
\end{CD}$$
shows that the kernel of $F-1$ on $\CH^n(X_{\K })_\Q$
is also independent of $\K$ if $F-1$ is an isomorphism on the left terms.
Moreover, the kernels and cokernels vanish on all other motivic cohomology 
groups, hence are independent of $\K$.

$2) \Rightarrow 1)$: By independence of $\K$ it suffices
to show the result for $\F$. Parshin's conjecture implies
that $H^i_\M(X,\Q(n))=H^i_\M(X_\F,\Q(n))^{F-1}=0$, hence the cokernel
also vanishes by the lemma. On the other hand, by a norm argument
$\CH_{\hom}^n(X)_\Q\cong \CH_{\hom}^n(X_{\F })_\Q^{F-1}$, and
this vanishes by Beilinson's conjecture. Then 
$(\CH_{\hom}^n(X_{\F })_\Q)_{F-1}$
vanishes by the lemma. 
\proofend

\section{Chow groups}
The Chow groups $\CH^n(X)$ have been studied extensively, and we 
are discussing our question in this setting. We have a filtration 
$$ 0\subset A^n(X)\subset \CH_{\hom}^n(X)\subset \CH^n(X),$$
where we let $\CH_{\hom}^n(X)$ be the cycles homologically equivalent to zero
for some fixed $\ell$-adic cohomology for $\ell\not=p$
and $A^n(X)$ are the cycles algebraically equivalent to zero,
i.e., $A^n(X)$ is the subgroup generated as follows: 
For any smooth connected curve $C/k$, two points 
$p_0,p_1\in C(k)$, and a cycle $Z\in \CH^n(X\times C)$, we let
$ Z|_{X\times\{p_0\}} -Z|_{X\times\{p_1\}}\in A^n(X)$.
The quotient $\Gr^n(X)=\CH_{\hom}^n(X)/A^n(X)$ is called the Griffiths
group.

If $p\geq 3$, then Fakhruddin \cite{fak-thesis}, 
see \cite[Remark A.3]{fak-comp}, shows that $\Gr^2(A)$ is of infinite 
rank for $A$ a generic abelian threefold
\footnote{Generic abelian varieties may not be defined over finite fields}.
In contrast, the following result of Soul\'e implies that for an abelian 
threefold $A$ over the algebraic closure of a finite field, 
$\CH_{\hom}^n(A)$ is torsion:

\begin{theorem}(Soul\'e \cite[Thm. 5]{soule})
Let $A$ be a product of curves or an abelian variety of dimension $d$ over 
the algebraic closure
of a finite field. Then for $n=0,1,d-1,d$, $\CH^n(A)$ is the direct sum 
of a finitely generated group and a torsion group. 
\end{theorem}

As can be seen from the Picard group, the Chow group depends
on the choice of the algebraically closed field. However, the
following classical result shows that the Griffiths group does not 
depend on the algebraically closed field:

\begin{theorem}\label{chowrigid}
If $X$ is smooth and projective over an algebraically
closed field $k$, then 
the group $\CH^n(X)/A^n(X)$ of cycles modulo algebraic equivalence, 
the Griffiths group, and $\CH^n(X)/\CH_{\hom}^n(X)$ 
are countable, and do not change when base extending to a larger
algebraically closed field. 
\end{theorem}

\proof 
Closed subschemes of $X$ are classified by the Hilbert scheme
$$\Hilb_{X/k}=\amalg_{\phi(x)\in \Q[x]}\Hilb_{X/k}^\phi$$
see \cite[Theorem 5.1]{nitsure}. 
More precisely, we have a bijection of sets
$$\Hom_k(T, \Hilb_{X/k})=\{Z\subset X\times_k T \text{ closed} \mid 
Z \text{ proper, flat} /T \},$$
and $\Hilb_{X/k}^\phi$ corresponds to the subschemes having Hilbert 
polynomials $\phi(x)$. There is a universal subscheme 
$U\subset X\times \Hilb_{X/k}$ (corresponding to the identity of 
$\Hilb_{X/k}$) such that the above bijection sends 
$f:T\to \Hilb_{X/k}$ to the pull-back of $U$ along $f$.
Each $\Hilb_{X/k}^\phi$ is projective, hence has finitely many connected components, so that 
$\Hilb_{X/k}$ has countably many connected components.
If two $k$-rational points $p_i$ of $\Hilb_{X/k}^\phi$, corresponding
to the closed subschemes $U_{p_i}$ of $X$, lie in the
same connected component, then there is a chain of irreducible curves 
$\iota: C\to \Hilb_{X/k}^\phi$ 
over $k$ passing through both points. Normalizing the curves
we can assume that they are smooth. We obtain a cartesian diagram 
$$ \begin{CD}
U_{p_i}@>>> X\times \{p_i\}@>>> \Spec k\\
@VVV @Vp_iVV @Vp_iVV\\
U_C@>>> X\times C@>>> C\\
@VVV @V\iota VV @V\iota VV\\
U@>>> X\times \Hilb_{X/k}^\phi@>>> \Hilb_{X/k}^\phi.
\end{CD}$$
Then $U_{p_i}=U_C|_{X\times \{p_i\}}$ and thus 
$U_{p_1}-U_{p_2}\in A^n(X)$,
i.e., any two subschemes corresponding to points in the same connected 
component are algebraically equivalent. 
It follows that $\CH^n(X)/A^n(X)$, and hence the Griffiths group 
as well as $\CH^n(X)/\CH_{\hom}^n(X)$, are countable.

Now let $\K/k$ be an extension of algebraically closed fields. The same
argument shows that every cycle on $X_\K$ is algebraically equivalent to
a cycle defined over $k$. Indeed, each component of
$\Hilb_{X_\K/\K}=\Hilb_{X/k}\times_k\K$ is the base change of a
connected component of $\Hilb_{X/k}$, and every such component contains
a $k$-rational point. Thus
$$\CH^n(X)/A^n(X)\longrightarrow\CH^n(X_\K)/A^n(X_\K)$$
is surjective and we show injectivity. 
Suppose that a class
$z\in\CH^n(X)$ becomes algebraically trivial over $\K$. An algebraic
equivalence between $z_\K$ and zero is represented by a family of cycles
over a connected curve together with two marked fibers. All the
data involved are of finite type and hence descend to a finitely
generated intermediate field $E$, $k\subset E\subset\K$. Spreading
out over an integral $k$-variety with function field $E$, and then
specializing at a $k$-rational point, gives an algebraic equivalence
between $z$ and zero over $k$. Hence $z\in A^n(X)$.

Next consider homological equivalence. For the fixed prime $\ell\ne p$, 
the smooth and proper base change gives an isomorphism
$
H^{2n}_{\et}(X,\Z_\ell(n))
\xrightarrow{\sim}
H^{2n}_{\et}(X_\K,\Z_\ell(n)),
$
compatible with the cycle class maps. Hence
$$
\CH^n(X)/\CH^n_{\hom}(X)
\longrightarrow
\CH^n(X_\K)/\CH^n_{\hom}(X_\K)
$$
is injective, and 
surjectivity follows by the above. 
Finally, the commutative diagram with exact rows
$$
\begin{CD}
0@>>>\Gr^n(X)@>>>\CH^n(X)/A^n(X)
@>>>\CH^n(X)/\CH^n_{\hom}(X)@>>>0\\
@.@VVV@VV\sim V@VV\sim V@.\\
0@>>>\Gr^n(X_\K)@>>>\CH^n(X_\K)/A^n(X_\K)
@>>>\CH^n(X_\K)/\CH^n_{\hom}(X_\K)@>>>0,
\end{CD}
$$
gives an isomorphism of the left terms.
\proofend

It is an interesting question whether the cone of $F-1$ on $A^n(X)$ 
is independent of $\K$.

\section{Cycle complexes and Gersten resolutions}
In this section we discuss Frobenius rigidity in terms of 
the cycle complex and Gersten resolutions. Let again $X$ be a
smooth and proper variety over a finite field $\F_q$.

\subsection*{Cycle complexes}
Consider Bloch's cycle complex $z^n(X,*)$, which in degree $i$
is the free abelian group on closed integral subschemes of codimension $n$
on $X\times \Delta^i$ meeting the faces properly. We are
interested in the fiber of
\begin{equation}\label{thecone}
z^n(X_\K,*)_\Q \stackrel{F-1}{\longrightarrow}z^n(X_\K,*)_\Q.
\end{equation}

Since we understand the situation with torsion coefficients,
we focus on rational coefficients in this section. 
If a point has a finite orbit, then the fiber of $F-1$ looks as follows:
$$ 0\to \Q\stackrel{\Delta}{\to} \bigoplus_{orbit}\Q 
\stackrel{F-1}{\longrightarrow}
\bigoplus_{orbit}\Q\stackrel{\Sigma}{\to} \Q\to 0,$$
where the map on the left is the diagonal embedding and the map on 
the right is the sum.
Thus we obtain that the kernel as well as the cokernel of $F-1$
restricted to the subcomplex consisting of points with finite orbits 
in \eqref{thecone} are isomorphic to $z^n(X,*)_\Q$, see Lemma \ref{finiteorbits}. 
For infinite orbits, we obtain 
$$\begin{CD}
0@>>> \bigoplus_\Z\Q @>F-1>>\bigoplus_\Z\Q@>\sum>> \Q@>>>0\\
\end{CD}$$

Thus the kernel of $F-1$ in \eqref{thecone} is isomorphic to $z^n(X,*)_\Q$,
and Frobenius rigidity is equivalent to the acyclicity of 
the complex consisting of 
one copy of $\Q$ for every infinite Frobenius orbit.

\subsection*{Gersten resolutions} 
We consider the following Frobenius rigidity statement for fields: 

{\it For every finitely generated field $L$ 
over $\F_q$, and any algebraically closed field extension, 
$\F \subseteq \K$,
the cones of $F-1$ acting on the second factor of the tensor product for
$R\Gamma(L\otimes_{\F_q} \F, \Q(n))$ and for $R\Gamma(L\otimes_{\F_q}\K,\Q(n))$
are quasi-isomorphic.}

\begin{proposition}
Frobenius rigidity for all smooth $X/\F_q$ is equivalent to
Frobenius rigidity for fields.
\end{proposition}

\proof
The field case can be derived from the smooth variety case
by taking the colimit over smaller and smaller open subsets of a model of $L$.
Conversely, filtering $X$ by dimension, the motivic cohomology of 
$X_\F$ and $X_\K$
can be calculated by a spectral sequence whose $E_1$-term is
$$
E_1^{a,b}=\bigoplus_{x\in X^{(a)}} H^{b-a}(k(x)\otimes_{\F_q}\K,\Q(n-a)) 
\Rightarrow H^{a+b}_\M(X_\K,\Q(n))$$
which reduces the rigidity question to fields.
\proofend

\subsection*{The ring \texorpdfstring{$L\otimes_k\K$}{L tensor K}}
\begin{theorem}[Sharp]\cite{sharp}
For any field $k$ and extension fields $L$ and $K$ of $k$, we have
$$\dim (L\otimes_kK) = \min\{\trdeg_k L, \trdeg_k K\}.$$
\end{theorem}

For any finitely generated field $L$ over $\F_q$,
the algebraic closure $k_0$ of $\F_q$ in $L$ is a finite
extension of $\F_q$. Hence we have
\[
L\otimes_{\F_q}\F \cong \prod_{\sigma:k_0\hookrightarrow\F}
L\otimes_{k_0,\sigma}\F,
\]
and each factor is a field, finitely generated over $\F$.
Thus we only need to consider finitely generated extension fields 
$L$ of an algebraically closed field $k$.

\begin{lemma}\cite[Prop.\ 4.3.9, Lemma 6.7.4.1]{EGAIV24}
For any algebraically closed field $k$, any finitely generated 
extension $L$ of $k$, and any field extension $\K$ of $k$,  
the ring $L\otimes_k\K$ is integral, regular, and essentially
of finite type over $\K$.
\end{lemma}

\begin{example}
We have 
$$k(t)\otimes_k\K=\{\frac{f(t)}{g(t)}\mid f(t)\in \K[t], 
g(t)\in k[t]-\{0\}\},$$
and the spectrum consists of $\P_\K^1$ with all the closed points of 
$\P_k^1$ removed. The Frobenius acts on the residue field $\K(t)$
of the generic point, and permutes the closed points in infinite orbits. 
\end{example}

We are going to calculate the weight one motivic cohomology
of $L\otimes_k\K$. First 
we have the following affine version of Rosenlicht's theorem:

\begin{proposition}\label{tensorunits}
Let $A$ and $B$ be algebras over an algebraically closed field $k$
such that $k$ is algebraically closed in $A$ and $B$.
Then there is an exact sequence
$$ 0\to k^\times \to A^\times \times B^\times \to 
(A\otimes_kB)^\times\to 0,$$
where the left map sends $c$ to $(c,c^{-1})$ and the right
map sends $(a,b)$ to $a\otimes b$.
\end{proposition}

By \cite[Lemma 1.1]{sweedler}, $k$ is algebraically closed in $A$ 
if and only 
if $A$ is reduced and has no non-trivial idempotent.

\proof
The non-trivial assertion is surjectivity of the right-hand map, 
i.e., that every unit
in $A\otimes_kB$ is of the form $a\otimes b$ for units $a,b$
of $A$ and $B$, respectively. This is a theorem of Sweedler
\cite[Thm. 1.2]{sweedler}. The exactness in the middle follows
from the following elementary Lemma.
\proofend

\begin{lemma}
Let $V$ and $W$ be vector spaces over a field $k$. If
$v\otimes w=v'\otimes w'$ for $v,w,v',w'\not=0$, 
then there is $s\in k^\times$ such that $v=sv'$ and $w=s^{-1}w'$.
\end{lemma}

\proof
Take a linear map $\varphi:V\to k$ such that $\varphi(v)=1$,
then for any $\psi:W\to k$ we have 
$\varphi\otimes \psi(v\otimes w)=\psi(w)$ and
$\varphi\otimes \psi(v'\otimes w')=\varphi(v')\psi(w')=
\psi(\varphi(v')w')$. In particular,
$\psi(w-\varphi(v')w')=0$ for all $\psi$, so that 
$w=cw'$, and similarly $v=dv'$ for $c,d\in k$.
Then $v\otimes w=cdv'\otimes w'=cdv\otimes w$,
hence $(cd-1)v\otimes w=0$ and $cd=1$.
\proofend

\begin{proposition}
Assume that the finitely generated extension field $L/k$ admits a smooth
and proper model $X$. Then we have a short exact sequence
$$0\to \Pic(X) \to \Pic(X_\K)\to \Pic(L\otimes_k\K)\to 0.$$
\end{proposition}

\proof
Let $Q$ be the quotient field of $L\otimes_k\K$, which 
agrees with the function field of $X_\K$. Then comparing
the short exact divisor sequence of $X_\K$ and $L\otimes_k\K$, 
we obtain the diagram
$$ \begin{CD}
0@>>> \K^\times @>>> Q^\times 
@>>> \bigoplus_{X_\K^{(1)}}\Z@>>> \Pic(X_\K)@>>>0\\
@.@VVV@|@VVV@VVV\\
0@>>> (L\otimes_k\K)^\times @>>> Q^\times 
@>>> \bigoplus_{new}\Z@>>> \Pic(L\otimes_k\K)@>>>0
\end{CD}$$
Here the sum runs over all divisors of $X_\K$ and all
divisors not coming as a base extension from a divisor of $X$,
respectively. A diagram chase then gives the exact sequence
$$ 0\to \K^\times \to (L\otimes_k\K)^\times \to \bigoplus_{X^{(1)}}\Z 
\stackrel{\tau}{\to} \Pic(X_\K)\to \Pic(L\otimes_k\K)\to 0.$$
But by Proposition \ref{tensorunits}, the first three terms 
are quasi-isomorphic to the short exact divisor sequence of $X$
$$ 0\to k^\times \to L^\times \to \bigoplus_{X^{(1)}}\Z$$
with cokernel $\Pic(X)$.
\proofend

{\bf Data Availability Statement:} 
No datasets were generated or analysed during this study.

\end{document}